\documentclass{amsart}

\usepackage{amsfonts}
\usepackage{latexsym}
\usepackage{amssymb}
\usepackage{amsmath}
\usepackage{color}
\usepackage{cite}

\usepackage{graphicx}

\def\relbd{\mathop\mathrm{relbd}\nolimits}
\def\relint{\mathop\mathrm{relint}\nolimits}
\def\bd{\mathop\mathrm{bd}\nolimits}
\def\conv{\mathop\mathrm{conv}\nolimits}
\def\aff{\mathop\mathrm{aff}\nolimits}
\def\lin{\mathop\mathrm{lin}\nolimits}

\def\abs#1{\left|#1\right|}

\def\R{\mathbb{R}}
\def\N{\mathbb{N}}
\def\K{\mathcal{K}}
\def\L{\mathcal{L}}
\def\E{\mathcal{E}}

\def\e{\mathrm{e}}

\def\cir{\mathrm{R}}

\def\inr{\mathrm{r}}

\newcommand{\D}{\mathrm{D}}

\def\Simp{\mathrm{S}}
\def\Cube{\mathrm{C}}
\def\Ball{\mathrm{B}}

\newtheorem{theorem}{Theorem}[section]
\newtheorem{lemma}{Lemma}[section]
\newtheorem{corollary}{Corollary}[section]
\newtheorem{proposition}{Proposition}[section]

\newtheorem{remark}{Remark}[section]

\numberwithin{equation}{section}

\begin{document}

\title[Improved Perel'man-Pukhov quotient]{Improving bounds for the Perel'man-Pukhov quotient for
inner and outer radii}

\author{Bernardo Gonz\'{a}lez Merino}
\address{Zentrum Mathematik, Technische Universit\"at M\"unchen, Boltzmannstr. 3, 85747 Garching bei M\"unchen, Germany}
\email{bg.merino@tum.de}

\subjclass[2010]{Primary 52A20, Secondary 52A40} \keywords{Inner and
outer radii, Perel'man-Pukhov inequality, Section and projection, Jung's inequality, Steinhagen's inequality}

\thanks{The author is partially supported by
Consejer\'ia de Industria,
Turismo, Empresa e Innovaci\'on de la CARM through Fundaci\'on
S\'eneca, Agencia de Ciencia y Tecnolog\'ia de la Regi\'on de
Murcia, Programa de Formaci\'on Postdoctoral de Personal
Investigador 19769/PD/15
and project 19901/GERM/15, Programme in Support of Excellence
Groups of the Regi\'{o}n de Murcia, and by MINECO project reference MTM2015-63699-P, Spain.}


\date{\today}
\begin{abstract}
In this work we study upper bounds for the ratio of
successive inner and outer radii of a convex
body $K$. This problem was studied by Perel'man
and Pukhov and it is a natural generalization of the classical
results of Jung and Steinhagen.
We also introduce a technique which relates sections and
projections of a convex body in an optimal way.
\end{abstract}
\maketitle

\section{Introduction}

The biggest radius of an $i$-dimensional Euclidean disc contained in
an $n$-dimensional convex body $K$ is denoted by $\inr_i(K)$,
whereas the smallest radius of a solid cylinder with $i$-dimensional
spherical cross-section containing $K$ is denoted by $\cir_i(K)$, for any $1\leq i\leq n$.
Perel'man in \cite{P} and independently Pukhov in \cite{Pu} studied
the relation between these inner and outer measures, and showed that
\begin{equation}\label{eq:PerPukh}
\frac{\cir_{n-i+1}(K)}{\inr_i(K)}\leq i+1,\quad 1\leq i\leq n.
\end{equation}
Unfortunately, the inequality is far from being best possible.
Two remarkable results in Convex Geometry are particular cases
of \eqref{eq:PerPukh}. Jung's inquality \cite{Ju} states
\begin{equation}\label{eq:Jung}
\frac{\cir_n(K)}{\inr_1(K)}\leq\sqrt{\frac{2n}{n+1}},
\end{equation}
and Steinhagen's inequality \cite{St} says
\begin{equation}\label{eq:Steinhagen}
\frac{\cir_1(K)}{\inr_n(K)}\leq\left\{\begin{array}{ll}\sqrt{n} & \text{if }n\text{ is odd,}\\
\frac{n+1}{\sqrt{n+2}} & \text{if }n\text{ is even.}
\end{array}\right.
\end{equation}
\eqref{eq:Jung} and \eqref{eq:Steinhagen} are best possible, since
the $n$-dimensional regular simplex $\Simp_n$ attains equality in both of them.
Therefore, it is natural
to conjecture that the regular simplex attains equality in the optimal upper
bound for the quotient given in \eqref{eq:PerPukh}.
If $i=1$ or $i=n$ the simplex $\Simp_n$ attains equality in
\eqref{eq:Jung} and \eqref{eq:Steinhagen}. If $i=2$ and $n$ is even, then
\[
\frac{\cir_{n-1}(\Simp_n)}{\inr_2(\Simp_n)}=\frac{(2n-1)\sqrt{3}}{\sqrt{2n(n+1)}},
\]
and in the remaining cases (c.f.~\cite{Br05}) it holds that
\[
\frac{\cir_{n-i+1}(\Simp_n)}{\inr_i(\Simp_n)}=\sqrt{1-\frac{i}{n+1}}\sqrt{i(i+1)}.
\]

In \cite{BH} the authors proved the reverse inequality
$\inr_i(K)\leq\cir_{n-i+1}(K)$, with equality for the Euclidean ball,
and moreover, Perel'man pointed out in \cite{P} that there exists
no constant $C>0$ fulfilling $\cir_{j}(K)\leq C\inr_i(K)$, for any
$1\leq i\leq n-1$ and $1\leq j\leq n-i$.

Perel'man improved \eqref{eq:PerPukh} when $n=3$ and $i=2$, by reducing
the bound $3$ down to $2.151$. The proof of the result, far from being trivial,
shows up hard to understand. In Section \ref{s:4}, we will give a comprehensive proof
of this inequality, as it has some interest by itself. The proof will also suggest
what kind of results would be desirable to be proven, in order to obtain
further improvements of this and other bounds.

Both proofs of \eqref{eq:PerPukh} in \cite{P,Pu} contain the hidden
result that for a simplex $\Simp_n\subset K$
of maximum volume in an $n$-dimensional convex body $K$, it holds
$\Simp_n\subset K\subset x+(n+2)\Simp_n$, where $x$ is the barycenter of $\Simp_n$. This directly
bounds the so-called Banach-Mazur distance (c.f.~\cite{Sch}) between $K$
and the class of simplices by $n+2$. This fact
has been independently proved in \cite{La}.

If $K$ is assumed to be a centrally symmetric set, Pukhov \cite{Pu} (see also
\cite{BoHe}) improved the inequality \eqref{eq:PerPukh}, by showing that
\begin{equation}\label{eq:PerPukh0sym}
\frac{\cir_{n-i+1}(K)}{\inr_i(K)}\leq\sqrt{\e}\min\{\sqrt{i},\sqrt{n-i+1}\},\quad 1\leq i\leq n,
\end{equation}
and it is neither best possible. In \eqref{eq:PerPukh0sym} $\e$ means the base of the natural logarithm.
In \cite{Go}, we improved the upper
bound when $n=3$ and $i=2$, from $\sqrt{2\e}$ down to $2$, but this inequality is still not best possible.
Indeed, it is conjectured that the $n$-dimensional cube
$\Cube_n$ and the regular crosspolytope $\Cube_n^\circ$
provide the biggest ratio in the inequality \eqref{eq:PerPukh0sym}. They fulfill
\begin{equation}
\frac{\cir_{n-i+1}(\Cube_n)}{\inr_i(\Cube_n)}=
\frac{\cir_{n-i+1}(\Cube_n^\circ)}{\inr_i(\Cube_n^\circ)}=
\sqrt{\frac{(n-i+1)i}{n}},\quad 1\leq i\leq n,
\end{equation}
(see \cite{Br05} and \cite{ESVW}).
Our first theorem, which follows from the main result in
Section \ref{s:2}, improves \eqref{eq:PerPukh0sym} in the 3-dimensional
case.
\begin{theorem}\label{th:0symm3dim}
For any centrally symmetric convex body $K\subset\R^3$, it holds that
\[
\frac{\cir_2(K)}{\inr_2(K)}\leq \frac{2\sqrt{2}}{\sqrt{3}}<1.633.
\]
\end{theorem}

In Section \ref{s:3}, we improve inequality \eqref{eq:PerPukh}
in some cases. Based on some ideas of Perel'man, we are able to show the following
theorem.
\begin{theorem}\label{th:Rn-1r2}
For any convex body $K\subset\R^n$, it holds that
\begin{equation}
\frac{\cir_{n-1}(K)}{\inr_2(K)}\leq 2\sqrt{2}\sqrt{\frac{n-1}{n}}.
\end{equation}
\end{theorem}
Moreover, we establish an improved bound for the case $i=n-1$.
\begin{theorem}\label{th:R2rn-1}
For any convex body $K\subset\R^n$, it holds that
\begin{equation}
\frac{\cir_2(K)}{\inr_{n-1}(K)}\leq 2\sqrt{2}\sqrt{n}.
\end{equation}
\end{theorem}
This result improves inequality \eqref{eq:PerPukh},
providing the right order in the dimension.

The outer radii $\cir_i(K)$ and the inner radii $\inr_i(K)$ have been extended
to arbitrary Minkowski spaces, i.e., finite dimensional
normed spaces (cf.~\cite{GK}). For the sake of completeness,
and although this paper is focused in the Euclidean metric,
we add a short section \ref{s:5} in which we provide a general upper bound
for the analogous quotient. Indeed, this bound
improves \eqref{eq:PerPukh0sym} in some cases.

For more information on the successive radii, their size for particular bodies
as well as computational aspects of these radii we refer to \cite{Ball,BH,BH2,Br05,BrKo,BrT06,GK,GK2,H92}.
Their relation with other measures have been studied in \cite{BH,HHC,HHC2},
their behavior with respect to other binary operations in \cite{CYL,GoHC12,GoHC14},
and their extensions to containers different from the Euclidean ball in \cite{GK,Ja}.
Moreover, quotients of different radii have been studied in \cite{BH2,BrKo,Go,GK,H92}.
We would like to point out that successive radii are particular
cases of the so-called Gelfand and Kolmogorov numbers in Banach Space Theory (cf.~\cite{CHP,GoHCH,Pin}),
and are widely used in Approximation Theory.


We now establish further notation. Let $\K^n$ denote the family of all convex bodies, i.e., compact convex sets, in
the $n$-dimensional Euclidean space $\R^n$, and we always assume
$K\in\K^n$. The subset of $\K^n$
consisting of all centrally (or $0$-) symmetric convex bodies, i.e.,
such that if $x=(x_1,\dots,x_n)^{\intercal}\in K$
then $-x\in K$, is denoted by $\K^n_0$. Let 
$|\,\cdot\,|_2$ be the standard 
Euclidean norm in $\R^n$ and $\Ball_n$ be the $n$-dimensional Euclidean unit ball.

The set of all $i$-dimensional linear subspaces of $\R^n$ is denoted by
$\L^n_i$. For the sake of brevity we denote by $\Ball_{i,L} = \Ball_n \cap L$
for any $L\in\L^n_i$.
We denote by $\lin(C)$, $\aff(C)$ and
$\conv(S)$, the linear, affine and convex hull of $C$, respectively, and we write $\relbd(C)$
to denote the relative boundary of any $C\subset\R^n$.
For any $x,y\in\R^n$, the line segment
with endpoints $x$ and $y$ is denoted by $[x,y]:=\conv(\{x,y\})$
We denote by $L^{\bot}$ and $u^\bot$
the orthogonal complement to $L$ and $\lin(\{u\})$, respectively, for any $L\in\L^n_i$ and $u\in\R^n$.
By $K|L$ we denote the orthogonal projection of $K$ onto
$L$. We use $\e_i$ for $i$-th canonical unit vector in
$\R^n$.

The width in the (unit) direction $u$, the diameter, the minimal width,
the circumradius and the inradius of $K$, all measured in the
Euclidean distance, are denoted by
$\omega(K,u)$, $\D(K)$, $\omega(K)$, $\cir(K)$ and $\inr(K)$,
respectively. For more information on these functionals and their
properties we refer to \cite[pp.~56--59]{BF}. Whenever
$K\in\K^n$ is contained in an affine subspace $x+L$, with $L\in\L^n_i$ and $x\in\R^n$,
we write $f(K;x+L)$ to denote that the functional $f$ has to be
evaluated with respect to the subspace $x+L$.
With this notation, the outer and inner measures
$\cir_i(K)$ and $\inr_i(K)$ can be expressed as
\begin{equation}\label{eq:definition}
\begin{array}{ll}
\displaystyle\cir_i(K)=\min_{L\in\L_i^n}\cir(K|L) & \text{ and }
\quad\displaystyle\inr_i(K)=\max_{L\in\L_i^n}\max_{x\in L^{\bot}}\inr(K\cap(x+L);x+L).
\end{array}
\end{equation}
Slightly modifying the definition of the inner radius $\inr_i(K)$, we obtain
another sequence of interior radii (\emph{cf}.~\cite{BH}, see also \cite{BHT}),
\[
\widetilde{\inr}_i(K):=\max_{L\in\L^n_i}\inr(K|L;L).
\]
These sequences of inner and outer measures extend the classic
radii, namely,
\begin{equation*}
\begin{split}
\cir_n(K) & =\cir(K),\quad \inr_n(K)=\widetilde{\inr}_n(K)=\inr(K),\\
\cir_1(K) & =\frac{\omega(K)}{2},\quad \inr_1(K)=\widetilde{\inr}_1(K)=\frac{\D(K)}{2}.
\end{split}
\end{equation*}
Moreover, the outer radii are increasing in $i$, whereas both
sequences of inner radii are decreasing in $i$, $1\leq i\leq n$. We also have that
$\inr_i(K)\leq\widetilde{\inr}_i(K)$, and for any $K\in\K^n_0$ and $1\leq i\leq n$,
then
\begin{equation}\label{eq:R_itildeinr_i}
\frac{\cir_{n-i+1}(K)}{\widetilde{\inr}_i(K)}\leq\sqrt{n-i+1}
\end{equation}
(see Theorem 1.3 in \cite{Go}).

\section{Centrally symmetric estimate}\label{s:2}

We first establish a lemma that will be needed in the proof of Theorem \ref{th:0symm3dim}.
This lemma reconstructs the largest disc contained in $K$, knowing in advance
that a projection of $K$ in a plane $L$ contains a disc of prescribed radius.
The main idea in the proof is to find six points in $K$ (three and their mirrored points in the origin),
such that they are all contained in a $2$-dimensional subspace and
their orthogonal projection onto $L$ forms a regular hexagon.
To do so, we build two sequences of six-tuples of points in $K$, and
we find the desired six-tuple as a limit of those sequences of six-tuples, using
a Bolzano-type argument.
\begin{lemma}\label{l:2-dim hexagon}
Let $K\in\K^3_0$, $L=\lin(\{\e_1,\e_2\})$ and
$\inr>0$ be such that $\inr B_{2,L}\subset K|L$. Then, there exist a
regular hexagon $\conv(\{\pm p_i:i=1,2,3\})$ inscribed in $\inr B_{2,L}$ and
points $\pm q_i\in K$, $i=1,2,3$, such that $\pm q_i|L=\pm p_i$,
$i=1,2,3$, and $\dim\conv(\{\pm q_i:i=1,2,3\})=2$.
\end{lemma}

\begin{proof}
For a fixed $u_1\in\relbd(\inr B_{2,L})$, we consider the regular hexagon
inscribed in $\inr B_{2,L}$ and having $u_1$ as a vertex, and call
$\overline{u}_1,\widetilde{u}_1$ the closest vertices to $u_1$.

Since $u_1,\overline{u}_1,\widetilde{u}_1\in K|L$, there exist points
$x_1^u,\overline{x}_1^u,\widetilde{x}_1^u\in K$ such that
\[
x_1^u|L=u_1,\quad\overline{x}_1^u|L=\overline{u}_1\quad\text{and}
\quad\widetilde{x}_1^u|L=\widetilde{u}_1.
\]
If $x^u_1\in\lin(\{\overline{x}^u_1,\widetilde{x}^u_1\})$, then $\conv(\{\pm
x^u_1,\pm \overline{x}^u_1,\pm \widetilde{x}^u_1\})$ is a 2-dimensional
convex body whose projection onto $L$ is the regular hexagon $\conv(\{\pm
u_1,\pm \overline{u}_1,\pm \widetilde{u}_1\})$. In this case, $p_1:=u_1$,
$p_2:=\overline{u}_1$, $p_3:=\widetilde{u}_1$, and $q_1:=x^u_1$,
$q_2:=\overline{x}^u_1$, $q_3:=\widetilde{x}^u_1$ show the lemma (\emph{cf.}~Figure \ref{f:cross}). So, we
assume $x^u_1\notin\lin(\{\overline{x}^u_1,\widetilde{x}^u_1\})$.

We observe that $x^u_1\in\lin(\{\overline{x}^u_1,\widetilde{x}^u_1\})$ if
and only if there exist $t,s\in\R$ such that
\[
t\bigl(\overline{u}_1,\overline{x}_{13}^u\bigr)^{\intercal}+s\bigl(\widetilde{u}_1,\widetilde{x}_{13}^u\bigr)^{\intercal}
=t\overline{x}^u_1+s\widetilde{x}^u_1=x^u_1=\bigl(u_1,x^u_{13}\bigr)^{\intercal},
\]
which holds if and only if $t\overline{u}_1+s\widetilde{u}_1=u_1$ and $t
\overline{x}^u_{13}+s\widetilde{x}^u_{13}=x^u_{13}$. Since
$u_1,\overline{u}_1,\widetilde{u}_1$ are consecutive vertices of a regular
hexagon, the unique solution of $t\overline{u}_1+s\widetilde{u}_1=u_1$ is
$t=s=1$. Therefore,
$x^u_1\notin\lin(\{\overline{x}^u_1,\widetilde{x}^u_1\})$ if and only if
$\overline{x}^u_{13}+\widetilde{x}^u_{13}\neq x^u_{13}$. We suppose
without loss of generality that
$\overline{x}^u_{13}+\widetilde{x}^u_{13}>x^u_{13}$. For the rest of the
proof we will use the same notation in the construction of the points,
namely: from any point $v\in\relbd(\inr B_{2,L})$, we derive
$\overline{v}$, $\widetilde{v}$, $x^v$, etc.

We write $w_1:=-u_1$. Then $\overline{w}_1=-\overline{u}_1$,
$\widetilde{w}_1=-\widetilde{u}_1$ and the symmetry of $K$ imply that
$x^w_1=-x^u_1$, $\overline{x}^w_1=-\overline{x}^u_1$,
$\widetilde{x}^w_1=-\widetilde{x}^u_1$, and thus
\[
\overline{x}^w_{13}+\widetilde{x}^w_{13}=-\overline{x}^u_{13}-\widetilde{x}^u_{13}<-x^u_{13}
=x^w_{13}.
\]
Let $u_2\in\relbd(\inr B_{2,L})$ be the ``midpoint'' on the circumference
$\relbd(\inr B_{2,L})$ between $u_1$ and $w_1$. If
$x^u_{23}=\overline{x}^u_{23}+\widetilde{x}^u_{23}$ then $p_1:=u_2$,
$p_2:=\overline{u}_2$, $p_3:=\widetilde{u}_2$, and $q_1:=x^u_2$,
$q_2:=\overline{x}^u_2$, $q_3:=\widetilde{x}^u_2$ show the lemma. If that
is not the case, then we can assume that
$\overline{x}^u_{23}+\widetilde{x}^u_{23}>x^u_{23}$ and define $w_2:=w_1$;
otherwise we just take $w_2$ to be the
midpoint and define $u_2:=u_1$. In the next step we take again the
midpoint $u_3=(u_2+w_2)/\abs{u_2+w_2}_2\in\relbd(\inr B_{2,L})$ and do the
same construction.

Iterating the process, either we find three points $p_i$, $i=1,2,3$,
verifying the required condition in some step, or we get two sequences
$(u_n)_n,(w_n)_n\subset\relbd(\inr B_{2,L})$, satisfying the following
properties:
\begin{itemize}\itemsep0pt
\item $d(u_n,w_n)=(1/2)d(u_{n-1},w_{n-1})$, where $d(a,b)$ is the length of the
shortest arc in $\relbd(\inr B_{2,L})$ joining the points
$a,b\in\relbd(\inr B_{2,L})$.
\item $\lim_{n\rightarrow\infty}u_n=\lim_{n\rightarrow\infty}w_n\in\relbd(\inr
B_{2,L})$. Let $p_1:=\lim_{n\rightarrow\infty}u_n$.
\item The vertices of the two corresponding hexagons sequences tend to the appropriate
limit, say
$\lim_{n\rightarrow\infty}\overline{u}_n=\lim_{n\rightarrow\infty}\overline{w}_n=:p_2$
and
$\lim_{n\rightarrow\infty}\widetilde{u}_n=\lim_{n\rightarrow\infty}\widetilde{w}_n=:p_3$.
\item $\overline{x}^u_{n3}+\widetilde{x}^u_{n3}>x^u_{n3}$ and
$\overline{x}^w_{n3}+\widetilde{x}^w_{n3}<x^w_{n3}$, for all $n\in\N$.
\end{itemize}
With this process, we also get sequences of points in $K$, namely
$(x^u_n)_n$, $(\overline{x}^u_n)_n$, $(\widetilde{x}^u_n)_n$, $(x^w_n)_n$,
$(\overline{x}^w_n)_n$ and $(\widetilde{x}^w_n)_n$. Since they are bounded
sequences (because they are contained in $K$), there exist convergent
subsequences in $K$ and we can suppose without loss of generality that
they are the same sequences. Thus
\begin{equation*}
\begin{split}
\lim_{n\rightarrow\infty}x^u_n & =x^u_0\in K,
    \quad\lim_{n\rightarrow\infty}\overline{x}^u_n=\overline{x}^u_0\in K,
    \quad\lim_{n\rightarrow\infty}\widetilde{x}^u_n=\widetilde{x}^u_0\in K,\\
\lim_{n\rightarrow\infty}x^w_n & =x^w_0\in K,
    \quad\lim_{n\rightarrow\infty} \overline{x}^w_n=\overline{x}^w_0\in K,
    \quad\lim_{n\rightarrow\infty}\widetilde{x}^w_n=\widetilde{x}^w_0\in K.
\end{split}
\end{equation*}
We observe that
\[
x^u_0|L=\left(\lim_{n\rightarrow\infty}x^u_n\right)|L=\lim_{n\rightarrow\infty}
(x^u_n|L)=\lim_{n\rightarrow\infty}u_n=p_1,
\]
and analogously,
\[
x^w_0|L=p_1,\quad \overline{x}_0^u|L=\overline{x}^w_0|L=p_2\quad\text{and}
\quad\widetilde{x}^u_0|L=\widetilde{x}^w_0|L=p_3.
\]
We notice also that
\[
\overline{x}^u_{03}+\widetilde{x}^u_{03}=\left(\lim_{n\rightarrow\infty}\overline{x}^u_n\right)_3+
\left(\lim_{n\rightarrow\infty}\widetilde{x}^u_n\right)_3=
\lim_{n\rightarrow\infty}\overline{x}^u_{n3}+\lim_{n\rightarrow\infty}\widetilde{x}^u_{n3}=
\lim_{n\rightarrow\infty}\bigl(\overline{x}^u_{n3}+\widetilde{x}^u_{n3}\bigr)\geq
\lim_{n\rightarrow\infty}x^u_{n3}=x^u_{03},
\]
and analogously, $\overline{x}^w_{03}+\widetilde{x}^w_{03}\leq x^w_{03}$.

If $\overline{x}^u_{03}+\widetilde{x}^u_{03}=x^u_{03}$ then the set of
points $q_1:=x_0^u$, $q_2:=\overline{x}_0^u$, $q_3:=\widetilde{x}_0^u$
together with $p_1,p_2,p_3$ show the lemma. Otherwise,
$\overline{x}^u_{03}+\widetilde{x}^u_{03}>x^u_{03}$. We observe that if
$\overline{x}^w_{03}+\widetilde{x}^u_{03}\leq x^u_{03}$ then the lemma is
proved: in fact, if this is the case, there exists $\lambda\in[0,1)$ such
that
\[
\bigl(\lambda
\overline{x}^u_0+(1-\lambda)\overline{x}^w_0\bigr)_3+\widetilde{x}^u_{03}=
\lambda
\overline{x}^u_{03}+(1-\lambda)\overline{x}^w_{03}+\widetilde{x}^u_{03}=x^u_{03},
\]
with
\[
\lambda \overline{x}^u_0+(1-\lambda)\overline{x}^w_0\in
K,\quad\bigl(\lambda\overline{x}^u_0+(1-\lambda)\overline{x}^w_0\bigr)|L=\lambda
p_1+(1-\lambda)p_1=p_1,
\]
and thus the set of points $q_1:=x^u_0$, $q_2:=\lambda
\overline{x}^u_0+(1-\lambda)\overline{x}^w_0$, $q_3:=\widetilde{x}^u_0$
shows the lemma.

So we assume that $\overline{x}^w_{03}+\widetilde{x}^u_{03}>x^u_{03}$.
Similarly, we now have that if
$\overline{x}^w_{03}+\widetilde{x}^w_{03}\leq x^u_{03}$, then there exists
$\lambda\in[0,1)$ such that
\[
\overline{x}^w_{03}+\bigl(\lambda
\widetilde{x}^u_0+(1-\lambda)\widetilde{x}^w_0\bigr)_3=
\overline{x}^w_{03}+\lambda
\widetilde{x}^u_{03}+(1-\lambda)\widetilde{x}^w_{03}=x^u_{03},
\]
and hence the set of points $q_1:=x^u_0$, $q_2:=\overline{x}^w_0$,
$q_3:=\lambda \widetilde{x}^u_0+(1-\lambda)\widetilde{x}^w_0$ shows the
lemma.

So we assume once more that this is not the case, i.e., that
$\overline{x}^w_{03}+\widetilde{x}^w_{03}>x^u_{03}$. But then, since
$\overline{x}^w_{03}+\widetilde{x}^w_{03}\leq x^w_{03}$ there exists
$\lambda\in[0,1)$ such that
\[
\overline{x}^w_{03}+\widetilde{x}^w_{03}=\lambda
x_{03}^u+(1-\lambda)x^w_{03}=\bigl(\lambda x_0^u+(1-\lambda)x^w_0\bigr)_3,
\]
and thus the points $q_1:=\lambda x_0^u+(1-\lambda)x^w_0$,
$q_2:=\overline{x}^w_0$, $q_3:=\widetilde{x}^w_0$ show the lemma.
\end{proof}

\begin{figure}
  \begin{center}
    \includegraphics[width=0.5\textwidth]{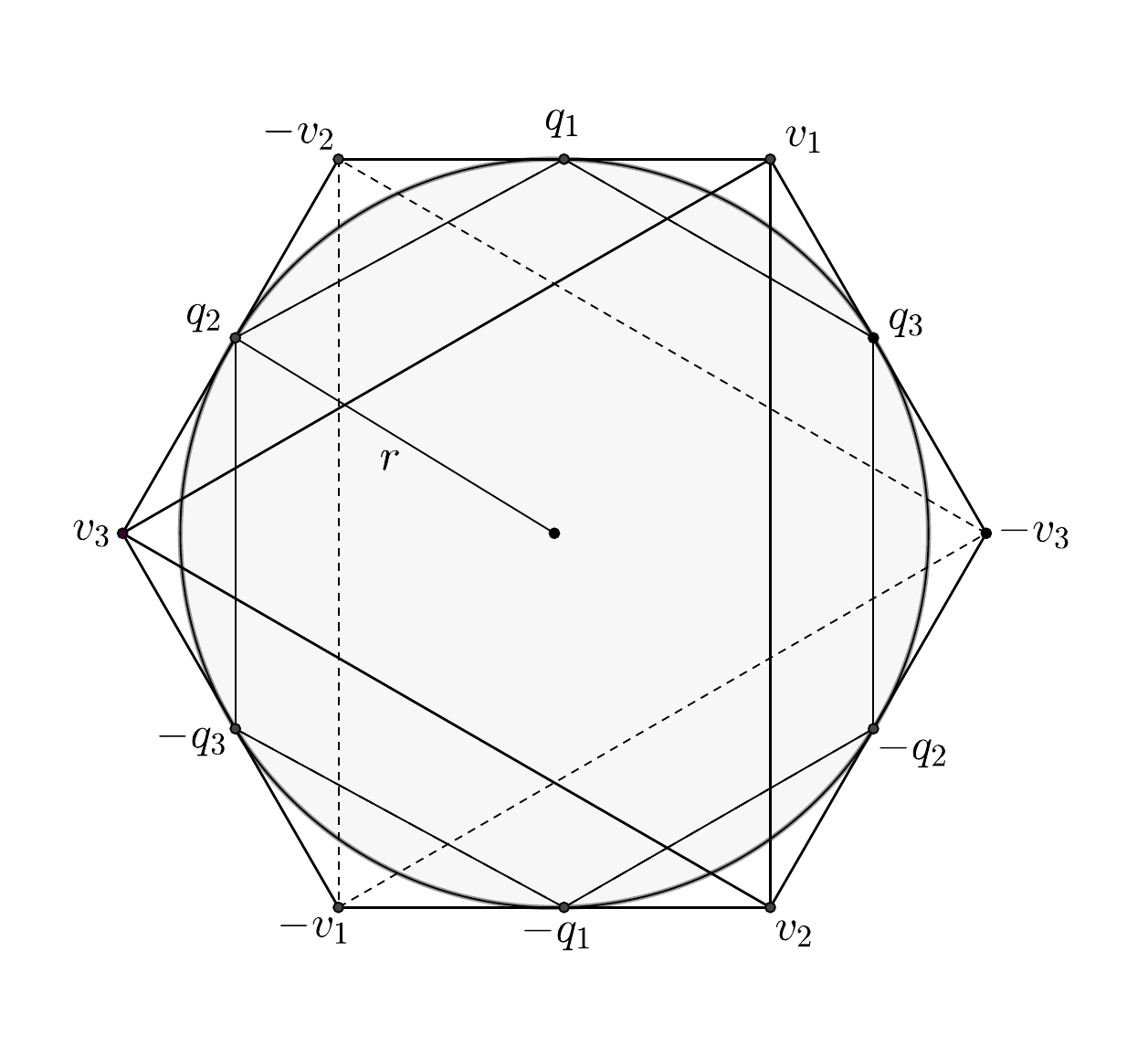}
    \caption{Upper view of the crosspolytope $P_{\varepsilon}:=\conv(\{\pm v_i:i=1,2,3\})$,
    where $v_1=(1/\sqrt{3},1,\varepsilon)^\intercal$, $v_2=(1/\sqrt{3},-1,\varepsilon)^\intercal$, $v_3=(-2/\sqrt{3},0,\varepsilon)^\intercal$,
    and $\varepsilon>0$.
    $P_\varepsilon$ has a hexagonal central section
    of vertices $\pm q_i$, $i=1,2,3$, and $r\Ball_{2,\e_3^\bot}\subset P_{\varepsilon}|\e_3^\bot$.}\label{f:cross}
  \end{center}
\end{figure}

Using Lemma \ref{l:2-dim hexagon}, we derive an inequality relating
$\inr_2(K)$ and $\widetilde{\inr}_2(K)$ for any $3$-dimensional set.
\begin{theorem}\label{th:inner_2}
Let $K\in\K^3_0$. Then
\[
\frac{\widetilde{\inr}_2(K)}{\inr_2(K)}\leq\frac{2}{\sqrt{3}}.
\]
The inequality is best possible.
\end{theorem}

\begin{proof}
By definition of $\widetilde{\inr}_2(K)$, there exists $L\in\L^3_2$ such
that $\widetilde{\inr}_2(K)=\inr(K|L;L)$. After a suitable rigid motion,
we can assume without loss of generality that $L=\lin(\{\e_1,\e_2\})$ and
that $\inr(K|L;L)B_{2,L}\subset K|L$. We now apply Lemma \ref{l:2-dim
hexagon} and find an inscribed regular hexagon
\[
H=\conv(\{\pm p_i:i=1,2,3\})\subset \inr(K|L;L) B_{2,L}
\]
and points $\pm q_i\in K$, $i=1,2,3$, such that
\[
\pm q_i|L=\pm p_i,\; i=1,2,3,\quad\text{ and }\quad \dim\conv(\{\pm
q_i:i=1,2,3\})=2.
\]
We call $C=\conv(\{\pm q_i:i=1,2,3\})$ and $L'=\lin C$. Then,
\[
\inr_2(K)\geq\inr(K\cap L';L')\geq\inr(C;L').
\]
We now show that $\inr(C;L')\geq\inr(H;L)$. Clearly,
\[
\inr(C;L')=\min_{x\in\relbd C}\abs{x}_2=\abs{x_0}_2
\]
for some $x_0\in\relbd C$. We can suppose that the points $q_1$ and $q_2$
are consecutive vertices and that $x_0=\lambda q_1+(1-\lambda)q_2$, for
some $\lambda\in(0,1)$. Since $q_j|L=p_j$, we have
$q_j=(p_j,q_{j3})^{\intercal}$, $j=1,2$, and~then
\[
\abs{x_0}^2_2=\bigl|\lambda q_1+(1-\lambda)q_2\bigr|^2_2=\bigl|\lambda
p_1+(1-\lambda)p_2\bigr|^2_2+\bigl|\lambda
q_{13}+(1-\lambda)q_{23}\bigr|^2\geq\bigl|\lambda
p_1+(1-\lambda)p_2\bigr|^2_2.
\]
The point $\lambda p_1+(1-\lambda)p_2\in\relbd H$, and therefore
\[
\bigl|\lambda p_1+(1-\lambda)p_2\bigr|_2\geq\min_{y\in\relbd
H}\abs{y}_2=\inr(H;L).
\]
From that, we get $\inr(C;L')=\abs{x_0}_2\geq\inr(H;L)$ and then
\[
\inr_2(K)\geq\inr(C;L')\geq\inr(H;L)=\frac{\sqrt{3}}{2}\,\widetilde{\inr}_2(K).
\]

It remains to be shown that the inequality is best possible. Let
$P_{\varepsilon}=\conv(\{\pm v_1,\pm v_2,\pm v_3\})$ be the non-regular
triangular antiprism in $\R^3$ with vertices
\[
v_1=\left(\frac{1}{\sqrt{3}},1,\varepsilon\right)^{\intercal},\quad
v_2=\left(\frac{1}{\sqrt{3}},-1,\varepsilon\right)^{\intercal},\quad
v_3=\left(-\frac{2}{\sqrt{3}},0,\varepsilon\right)^{\intercal},
\]
$\varepsilon>0$ (see Figure \ref{f:cross}). In pg. 10 and Figure 1 of \cite{GoHC12} it was shown that
$\inr_2(P_{\varepsilon})=\sqrt{3}/2$ for $\varepsilon$ small enough.
Since the set $P_{\varepsilon}|\lin(\{\e_1,\e_2\})$ is a regular hexagon
with $2$-dimensional inradius $1$, then
$\widetilde{\inr}_2(P_{\varepsilon})\geq 1$. Therefore
\[
1\leq\widetilde{\inr}_2(P_{\varepsilon})\leq\frac{2}{\sqrt{3}}\,\inr_2(P_{\varepsilon})=
\frac{2}{\sqrt{3}}\frac{\sqrt{3}}{2}=1,
\]
and thus
$\widetilde{\inr}_2(P_{\varepsilon})=(2/\sqrt{3})\,\inr_2(P_{\varepsilon})$.
\end{proof}

\begin{proof}[Proof of Theorem \ref{th:0symm3dim}]
Particularizing \eqref{eq:R_itildeinr_i} in $n=3$ and $i=2$,
together with Theorem \ref{th:inner_2}, we get that
\[
\frac{\cir_2(K)}{\inr_2(K)}=\frac{\cir_2(K)}{\widetilde{\inr}_2(K)}
\frac{\widetilde{\inr}_2(K)}{\inr_2(K)}\leq\sqrt{2}\frac{2}{\sqrt{3}}.\qedhere
\]
\end{proof}

Before concluding this section, we leave to the reader the analogous statement
to Lemma \ref{l:2-dim hexagon} and Theorem \ref{th:inner_2} for non-symmetric
convex sets.
\begin{lemma}\label{lem:general3}
Let $K\in\K^3$, $L=\lin(\{\e_1,\e_2\})$ and
$\inr>0$ be such that $\inr B_{2,L}\subset K|L$. Then, there exist a
square $\conv(\{\pm p_i:i=1,2\})$ inscribed in $\inr B_{2,L}$ and
points $q_{i,\pm}\in K$, $i=1,2$, such that $q_{i,\pm}|L=\pm p_i$,
$i=1,2$, and $\dim\conv(\{q_{i,\pm}:i=1,2\})=2$.
\end{lemma}

\begin{theorem}\label{th:general3}
Let $K\in\K^3$. Then
\[
\frac{\widetilde{\inr}_2(K)}{\inr_2(K)}\leq\sqrt{2}.
\]
The inequality is best possible.
\end{theorem}

\begin{remark}
In order to prove Lemma \ref{lem:general3}, it would be sufficient
to find an inscribed square, s.t.~the segments $[q_{1,+},q_{1,-}]$ and $[q_{2,+},q_{2,-}]$
intersect in their midpoints, i.e., if $(q_{1,+})_3+(q_{1,-})_3=(q_{2,+})_3+(q_{2,-})_3$.

Equality holds in Theorem \ref{th:general3} for a simplex
with vertices
\[
\left(\pm 1,0,\varepsilon\right)^\intercal\quad\text{and}\quad\left(0,\pm 1,-\varepsilon\right)^\intercal,
\]
for small enough $\varepsilon>0$.

Let us also remark that doing the same as in Theorem \ref{th:0symm3dim},
for $K\in\K^3$, i.e.~applying Theorem \ref{th:general3} and Proposition 2.1 in \cite{Go},
would imply that
\[
\frac{\cir_2(K)}{\inr_2(K)}=\frac{\cir_2(K)}{\widetilde{\inr}_2(K)}
\frac{\widetilde{\inr}_2(K)}{\inr_2(K)}\leq 3,
\]
still worse than the best known bound 2.151.
\end{remark}

\section{Improved general upper bounds}\label{s:3}

In the proof of Theorem \ref{th:Rn-1r2}, we extend some ideas
of Perel'man \cite{P}, slightly modifying some steps.

\begin{proof}[Proof of Theorem \ref{th:Rn-1r2}]
After a suitable translation of $K$, we can suppose that the diameter of
$K$ is given by $\D(K)=2\abs{p}_2$ for $p,-p\in K$. Let $p_1,p_2\in
K|p^{\bot}$ be such that $\abs{p_1-p_2}_2=\D(K|p^{\bot})$. We are going to
prove that
\begin{equation}\label{eq:D<4r2}
\D(K|p^{\bot})\leq 4\inr_2(K).
\end{equation}
So, we assume the contrary, $\D(K|p^{\bot})>4\inr_2(K)$, and we will get a
contradiction. Let $q_1,q_2\in K$ be such that $q_j|p^{\bot}=p_j$, for
$j=1,2$, and we write

\[
P=\conv\left(\left\{\frac{1}{2}(p+q_j),\frac{1}{2}(-p+q_j):j=1,2\right\}\right)\subset
K.
\]

We first observe that $P$ is a (2-dimensional) parallelogram, because
\begin{align}
\frac{1}{2}(p+q_1) & -\frac{1}{2}(p+q_2)
    =\frac{1}{2}(q_1-q_2)=\frac{1}{2}(-p+q_1)-\frac{1}{2}(-p+q_2)\quad\text{ and}\nonumber\\[1mm]
\frac{1}{2}(p+q_1) & -\frac{1}{2}(-p+q_1)
    =p=\frac{1}{2}(p+q_2)-\frac{1}{2}(-p+q_2),\label{RS eq:P_2-parallel}
\end{align}
and since $P$ is a 0-symmetric convex body, $\inr(P;\aff(P))=\omega(P;\aff(P))/2$.

Next we compute the width $\omega(P;\aff(P))$. Let $h,h'$ denote the
heights of the parallelogram $P$ corresponding to the edges
$\bigl[(p+q_1)/2,(p+q_2)/2\bigr]$ and $\bigl[(p+q_1)/2,(-p+q_1)/2\bigr]$,
respectively. 
From \eqref{RS eq:P_2-parallel} we get, on the one hand, that $h$ is just the distance
between the orthogonal projections onto $p^{\bot}$ of the points
$(p+q_1)/2$ and $(p+q_2)/2$, i.e., the distance between $p_1/2$ and
$p_2/2$. Thus, $h'=\abs{p_1-p_2}_2/2=\D(K|p^{\bot})/2$. On the other hand,
since
\[
\frac{\abs{\frac{p+q_1}{2}-\frac{-p+q_1}{2}}_2}{h}=\frac{\abs{\frac{p+q_1}{2}-\frac{p+q_2}{2}}_2}{h'},
\]
then we have
\[
h=\frac{2h'\abs{p}_2}{\abs{q_1-q_2}_2}=\frac{h'\,\D(K)}{\abs{q_1-q_2}_2}\geq
h',
\]
where the inequality comes from the fact that $q_1,q_2\in K$ and then
$\abs{q_1-q_2}_2\leq\D(K)$. Therefore
\[
\omega(P;\aff(P))=\min\{h,h'\}=h'=\frac{\D(K|p^{\bot})}{2},
\]
and hence
\[
\inr(K\cap\aff(P);\aff(P))\geq\inr(P;\aff(P))=\frac{\omega(P;\aff
(P))}{2}=\frac{\D(K|p^{\bot})}{4}>\inr_2(K),
\]
a contradiction.

This shows \eqref{eq:D<4r2}, and then, applying Jung's inequality \eqref{eq:Jung}
to the ($n-1$)-dimensional convex body $K|p^{\bot}$, we finally
get that
\[
\cir_{n-1}(K)\leq\cir(K|p^{\bot})\leq\sqrt{\frac{n-1}{2n}}\D(K|p^{\bot})\leq
2\sqrt{2}\sqrt{\frac{n-1}{n}}\inr_2(K).
\]
\end{proof}


For the proof of Theorem \ref{th:R2rn-1}, we need to remember (see \cite{GK}) that for
every $K\in\K^n$, there exist $x,y\in K$ s.t.
\[
\omega(K)=\omega(K|\aff([x,y]);\aff([x,y]))=\omega\left(K,\frac{x-y}{|x-y|_2}\right)=|x-y|_2.
\]
\begin{proof}[Proof of Theorem \ref{th:R2rn-1}]
After a suitable rigid motion of $K$, we can suppose that
$\pm(\omega(K)/2)\e_2\in K$ and $K$ is contained
between the parallel supporting hyperplanes $\pm(\omega(K)/2)\e_2+\e^{\bot}_2$.
Our aim is to show that
$\omega(K\cap\e_2^{\bot};\e_2^{\bot})\geq(1/\sqrt{2})\cir_2(K)$.
After rotating $K$ around $\lin(\{\e_2\})$, we can furthermore
assume that $\omega(K\cap\e_2^{\bot};\e_2^{\bot})=|x-y|_2$, with $x,y\in K\cap\e_2^{\bot}$ and $x-y\in\lin(\{\e_1\})$.
Moreover, let $L_x,L_y\in\L^{n-1}_{n-2}$ be two parallel supporting $(n-2)$-planes
of $K\cap\e_2^{\bot}$ in $x$ and $y$; respectively, s.t.~$L_x,L_y\subset\e_2^{\bot}$.
Then, there exist $H_x,H_y\in\L^n_{n-1}$ two (non-necessarily parallel) supporting
hyperplanes of $K$ in $x$ and $y$, respectively, and s.t.~$L_x\subset H_x$ and $L_y\subset H_y$.
Therefore, the outer normals of $K$ in $x$ and $y$ are vectors
$a_1\e_1+a_2\e_2$ and $b_1\e_1+b_2\e_2$, respectively, where $a_1,b_1,a_2,b_2\in\R$.
Let us denote $\omega:=\omega(K)$ and $\omega':=\omega(K\cap\e_2^{\bot};\e_2^{\bot})$.

We observe that $K|\lin(\{\e_1,\e_2\})$ is contained in the
trapezoid determined by the hyperplanes
\begin{equation}\label{eq:trapezoid}
\pm(\omega/2)\e_2+\lin(\{\e_1\}),\quad x+(a_1\e_1+a_2\e_2)^{\bot},\quad y+(b_1\e_1+b_2\e_2)^{\bot}.
\end{equation}
Moreover, let $a\e_1:=x|\lin(\{\e_1,\e_2\})$ and $-b\e_1:=y|\lin(\{\e_1,\e_2\})$, $a,b\geq 0$,
and $a+b=\omega'$.

We now show that $K|\lin(\{\e_1,\e_2\})$ is contained on the left hand side of the line $2a\e_1+\lin(\{\e_2\})$.
Indeed, the supporting line $a\e_1+(a_1\e_1+a_2\e_2)^{\bot}$ hits
$\pm(\omega/2)\e_2+\lin(\{\e_1\})$ in $(a\pm t)\e_1\pm(\omega/2)\e_2$, respectively, for some $t\in\R$.
Moreover, since $\pm(\omega/2)\e_2\in K$, then $a\pm t\geq 0$,
from which $t\in[-a,a]$. Since $K|\lin(\{\e_1,\e_2\})$
is contained in the trapezoid given by the lines \eqref{eq:trapezoid},
the most-right point is given by one of the vertices $(a\pm t)\e_1\pm(\omega/2)\e_2$,
and whose first coordinate is bounded by $a\pm t\leq 2a$, proving the assertion.
By an analogous argument, $K|\lin(\{\e_1,\e_2\})$ is on the right hand side of the line
$-2b\e_1+\lin(\{\e_2\})$.

This shows that $K|\lin(\{\e_1,\e_2\})$ is contained
in a box of length $\omega$ in the direction $\e_2$
and length $2a+2b=2\omega'$ in the direction $\e_1$.
This immediately implies that $\omega\leq 2\omega'$ (otherwise,
$\omega(K|\lin(\{\e_1\});\lin(\{\e_1\}))\leq 2\omega'<\omega$, a contradiction).
Since the circumradius of this box is $\sqrt{(\omega')^2+(\omega/2)^2}$,
then
\[
\cir(K|\lin(\{\e_1,\e_2\}))\leq\sqrt{(\omega')^2+(\omega/2)^2}\leq\sqrt{2}\omega'.
\]
Moreover, since $\lin(\{\e_1,\e_2\})\in\L^n_2$, then
$\cir_2(K)\leq\cir(K|\lin(\{\e_1,\e_2\}))$,
which together with the above, finally shows that
\begin{equation}\label{eq:R_2}
\cir_2(K)\leq\sqrt{2}\omega'.
\end{equation}

By Steinhagen's inequality \eqref{eq:Steinhagen} applied to $K\cap\e_2^{\bot}$,
and since $\sqrt{n-1},n/\sqrt{n+1}\leq \sqrt{n}$, then
\[
\omega'=\omega(K\cap\e_2^{\bot};\e_2^{\bot})\leq 2\sqrt{n}\inr(K\cap\e_2^{\bot};\e_2^{\bot}).
\]
This, together with \eqref{eq:R_2}, imply that
$\cir_2(K)\leq 2\sqrt{2}\sqrt{n}\inr_{n-1}(K)$, concluding the proof.
\end{proof}

It is not clear whether Theorem \ref{th:Rn-1r2} or Theorem \ref{th:R2rn-1}
induce for $n\geq 4$ tight inequalities or not.

\section{Perel'man's inequality}\label{s:4}

This section is devoted to show a comprehensive proof of
Perel'man's inequality $\cir_2(K)/\inr_2(K)\leq 2.151$. Since it uses
some hidden results, we establish them here. Some of them are well-known,
while others cannot be found in the literature.

Santal\'o in \cite{Sa}, his famous work on complete systems of
inequalities, proved that
\begin{equation}\label{eq:Santalo}
2R(K) \left(2R(K) + \sqrt{4R(K)^2-D(K)^2}\right)r(K) \ge D(K)^2\sqrt{4R(K)^2-D(K)^2},
\end{equation}
for any $K\in\K^2$. Moreover, equality holds if and only if $K$ is an
isosceles triangle, with two longer sides of equal length.

Next result is a characterization by touching points for the circumradius
of $K$. Remember that we address here the Euclidean case, but this characterization
is well-known even when the ball is an arbitrary convex body (c.f.~\cite{BrKo}).
\begin{proposition}\label{prop:ocuh}
Let $K\in\K^n$ be s.t.~$K\subset\Ball_n$. The following are equivalent:
\begin{itemize}
\item $\cir(K)=1$.
\item There exist $p^1,\dots,p^j\in K\cap\bd\Ball_n$, $2\leq j\leq n+1$,
s.t.~$0\in\conv(\{p^1,\dots,p^j\})$. In particular, $\cir(\conv(\{p^1,\dots,p^j\}))=1$.
\end{itemize}
\end{proposition}

\begin{lemma}\label{l:tildeinr}
Let $K\in\K^i$ be embedded in $\R^n$, and let $L\in\L^n_i$, where $1\leq i\leq n$. Then
$\inr(K|L;L)\leq\inr(K;\aff(K))$.
\end{lemma}

\begin{proof}
Let us define $\inr:=\inr(K|L;L)$. After a suitable rigid motion of $K$, we can suppose that
$L=\lin(\{\e_1,\dots,\e_i\})$ and $\inr\Ball_{i,L}\subset K|L$. Furthermore, for
every $u\in\relbd(\inr\Ball_{i,L})$, there exist $p^u_{i+1},\dots,p^u_n\in\R$, s.t.
\[
p^u:=u+(0,\dots,0,p^u_{i+1},\dots,p^u_n)\in K.
\]

Moreover, for the point $p:=(1/2)(p^u+p^{-u})\in K$, with $u\in\relbd(\inr\Ball_{i,L})$, we have that
$p|L=0$.

If $\inr=0$ or $\dim(K|L)<i$, the assertion immediately follows, thus let us assume $\inr>0$ and $\dim(K|L)=i$.
For every $q\in K|L$, let $p_q\in K$ be s.t.~$p_q|L=q$, and observe that $(p^q+p^{-q})|L=0$
yields $p^0=(1/2)(p^q+p^{-q})$, for every $q\in\relint(\inr\Ball_{i,L})$. Therefore,
\[
|p^0-p^u|_2^2=|u|_2^2+|p^0_{i+1}-p^u_{i+1}|^2+\cdots+|p^0_n-p^u_n|^2\geq|u|_2^2=\inr^2,
\]
for every $u\in\relbd(\inr\Ball_{i,L})\subset K|L$, hence $(p^0+\inr\Ball_{i,\aff(K)})\subset K$,
and thus we conclude that $\inr(K;\aff(K))\geq\inr$, finishing the lemma.
\end{proof}

Next corollary is the analogous statement to Lemma 3.1 in \cite{Go}
(and Lemma \ref{l:2-dim hexagon} and Theorem \ref{th:general3}, too)
when $K$ is not necessarily symmetric, and bounds $\widetilde{\inr}_i(K)$
from above in terms of $\inr_i(K)$.
\begin{corollary}\label{cor:inner22}
Let $K\in\K^n$ and $1\leq i\leq n$. Then
$\widetilde{\inr}_i(K)\leq i\inr_i(K)$.
The inequality is best possible when $i=1$.
\end{corollary}

\begin{proof}
After a suitable rigid motion we can suppose that there exists
$L\in\L^n_i$ such that
\[
\widetilde{\inr}_i(K)B_{i,L}\subset K|L.
\]
We take points $p_1,\dots,p_{i+1}\in\relbd(\widetilde{\inr}_i(K)B_{i,L})$
being the vertices of an $i$-dimensional regular simplex of $L$,
$\Simp_i=\conv(\{p_j:j=1,\dots,i+1\})$. There exist points $q_1,\dots,q_{i+1}\in
K$ such that $q_j|L=p_j$, $j=1,\dots,i+1$, and we define
$\Simp_i'=\conv(\{q_j:j=1,\dots,i+1\})\subset K$. By Lemma \ref{l:tildeinr},
we have that $\inr(\Simp_i';\aff(\Simp_i'))\geq\inr(\Simp_i'|L;L)=\inr(\Simp_i;L)$.
Since $\Simp_i$ is an $i$-dimensional regular simplex, then $\cir(\Simp_i;\aff(\Simp_i))=i\inr(\Simp_i;\aff(\Simp_i))$,
and hence
\[
\widetilde{\inr}_i(K)=\cir(\Simp_i;\aff(\Simp_i))=
i\inr(\Simp_i;\aff(\Simp_i))\leq i\inr(\Simp_i';\aff(\Simp_i')).
\]
Observe that $\Simp_i'\subset K$ implies $\inr(\Simp_i';\aff(\Simp_i'))\leq\inr(K\cap\aff(\Simp_i');\aff(\Simp_i'))\leq\inr_i(K)$,
because $\aff(\Simp_i')$ is an $i$-dimensional affine subspace, and therefore we conclude
$\widetilde{\inr}_i(K)\leq i\inr_i(K)$.
\end{proof}

\begin{proposition}\label{prop:perelman}
Let $K\in\K^3$. Then $\cir_2(K)/\inr_2(K)\leq 2.151$.
\end{proposition}

\begin{proof}
After a suitable translation of $K$, we can suppose
that $0,p\in K$ are s.t.~$\D([0,p])=\D(K)$. In the proof
of Theorem \ref{th:Rn-1r2} we showed (see \eqref{eq:D<4r2})
that $\D(K|p^{\bot})\leq 4\inr_2(K)$.

Using Proposition \ref{prop:ocuh}, there exist points $p_1,p_2,p_3\in K|p^{\bot}$,
vertices of the simplex $\Simp:=\conv(\{p^1,p^2,p^3\})$,
s.t.~$\cir(\Simp)=\cir(K|p^{\bot})$. Since $\Simp\subset K|p^{\bot}$,
then $\D(\Simp)\leq\D(K|p^{\bot})$ and $\inr(\Simp;p^{\bot})\leq \inr(K|p^{\bot};p^{\bot})$.

Since $\Simp$ is planar, using \eqref{eq:Santalo}, we have that
\[
2\cir(\Simp) \left(2\cir(\Simp) + \sqrt{4\cir(\Simp)^2-\D(\Simp)^2}\right)\inr(\Simp;p^{\bot}) \ge \D(\Simp)^2\sqrt{4\cir(\Simp)^2-\D(\Simp)^2}.
\]
Now, we solve this inequality in $\D(\Simp)$. To do so, we normalize it in terms
of $x:=r(\Simp;p^{\bot})/\cir(\Simp)$ and $y:=\D(\Simp)/\cir(\Simp)$. The only sharp valid inequality, can be
easily found by using the fact that \eqref{eq:Santalo} reaches equality
for isosceles triangles:
\[
y\geq\sqrt{2}\sqrt{x+1+\sqrt{1-2x}}.
\]
Therefore, we derive that
\[
\sqrt{2}\sqrt{\frac{\inr(\Simp;p^{\bot})}{\cir(\Simp)}+1+\sqrt{1-2\frac{\inr(\Simp;p^{\bot})}{\cir(\Simp)}}}\leq
\frac{\D(\Simp)}{\cir(\Simp)}\leq \frac{\D(K|p^{\bot})}{\cir(\Simp)}\leq4\frac{\inr_2(K)}{\cir(\Simp)}.
\]

Let $q_1,q_2,q_3\in K$ be s.t.~$q_i|p^{\bot}=p_i$, $i=1,2,3$, and let $\Simp':=\conv(\{q^1,q^2,q^3\})$.
Lemma \ref{l:tildeinr} implies $\inr(\Simp;p^{\bot})\leq \inr(\Simp';\aff(\Simp'))$, and
since $\Simp'\subset K$, then $\inr(\Simp;p^{\bot})\leq \inr(K\cap\aff(\Simp');\aff(\Simp'))\leq \inr_2(K)$.

Moreover, the function $\sqrt{x+1+\sqrt{1-2x}}$ is decreasing
in $x\in[0,1/2]$, which is the range of possible values of $\inr(\Simp;p^{\bot})/\cir(\Simp)$.
Hence, we obtain that
\[
\sqrt{2}\sqrt{\frac{\inr_2(K)}{\cir(\Simp)}+1+\sqrt{1-2\frac{\inr_2(K)}{\cir(\Simp)}}}\leq
4\frac{\inr_2(K)}{\cir(\Simp)}.
\]
Solving this in $\inr_2(K)/\cir(\Simp)$, is a nasty polynomial of degree four. Using some
Algebraic tool, we can get that
\[
\frac{\inr_2(K)}{\cir(\Simp)}\gtrapprox 0.46498.
\]
The inverse of this number $0.46498$ is exactly the mysterious Perel'man number $2.15063$.
Since $\cir_2(K)\leq\cir(K|p^{\bot})=\cir(\Simp)$, we conclude $\cir_2(K)/\inr_2(K)\leq 2.151$.
\end{proof}

\begin{remark}
The proof of Proposition \ref{prop:perelman} shows that it would be desirable to
extend inequality \eqref{eq:Santalo} to higher dimensions. It may not only improve
the best known bounds of \eqref{eq:PerPukh}, but would also complete
the corresponding Blaschke-Santal\'o diagram for
the functionals $\inr,\D,\cir$ in $\R^n$ (c.f.~\cite{BrGo,HCS,Sa}).
\end{remark}

\section{Perel'man-Pukhov quotient in Minkowski spaces}\label{s:5}

Let us denote by $(\R^n,||\,\cdot\,||)$ an $n$-dimensional Minkowski space,
and its unit ball by $\Ball=\{x\in\R^n:||x||\leq 1\}$.
We denote by $\cir_i(K,\Ball)$ the smallest $\rho\geq 0$
s.t.~$K\subset x+\rho(\Ball+L)$, for some $L\in\L^n_{n-i}$, $x\in\R^n$ and $1\leq i\leq n$. Analogously, we denote
by $\inr_i(K,\Ball)$ the biggest $\rho\geq 0$ s.t.~$x+\rho(\Ball\cap L)\subset K$, for some $L\in\L^n_i$, $x\in \R^n$ and $1\leq i\leq n$.
Both functionals are increasing and homogeneous of degree $1$
in the first entry, whereas they are decreasing and homogeneous of degree $-1$ in the second one.
They extend the inner and outer radii in the Euclidean setting, i.e.,
$\cir_i(K,\Ball_2)=\cir_i(K)$ and $\inr_i(K,\Ball_2)=\inr_i(K)$, $1\leq i\leq n$.

\eqref{eq:Jung} and \eqref{eq:Steinhagen} have their counterparts in Minkowski
spaces, and they state that
\[
\frac{\cir_n(K,\Ball)}{\inr_1(K,\Ball)}\leq\frac{2n}{n+1}\quad\text{and}\quad\frac{\cir_1(K,\Ball)}{\inr_n(K,\Ball)}\leq \frac{n+1}{2},
\]
and are known as Bohnenblust \cite{Bo} and Leichtweiss \cite{Le} inequality, respectively.

John's theorem \cite{Jo} states that for any $K\in\K^n$ we have that
$\E\subset x+K\subset n\E$, for some $x\in\R^n$, where $\E$ is the ellipsoid
of maximum volume contained in $x+K$, called John's ellipsoid. Moreover, if $K\in\K^n_0$,
we can replace the value $n$ by $\sqrt{n}$ and assume that $x=0$.
We say that $K$ is in John's position if $\Ball_2$ is the John's ellipsoid of $K$.

We always assume that an ellipsoid $\E$ is centered in the origin, i.e.,
$\E=f(\Ball_2)$, for some non-singular linear application $f$.
In \cite{H91} it was shown that for any ellipsoid $\E\in\K^n$, we have that
$\cir_{n-i+1}(\E)=\inr_i(\E)$, for every $1\leq i\leq n$.

\begin{lemma}\label{l:affine}
Let $\Ball_j\in\K^n_0$, $j=1,2$, and let $f$ be a non-singular linear application.
Then $\cir_i(\Ball_1,\Ball_2)=\cir_i(f(\Ball_1),f(\Ball_2))$ and
$\inr_i(\Ball_1,\Ball_2)=\inr_i(f(\Ball_1),f(\Ball_2))$, $1\leq i\leq n$.
\end{lemma}

\begin{proof}
We have that
\[
\Ball_1\subset \rho\Ball_2+L\quad\text{if and only if}\quad f(\Ball_1)\subset \rho f(\Ball_2)+f(L),
\]
as well as
\[
\rho\Ball_1\cap L\subset\Ball_2\quad\text{if and only if}\quad \rho f(\Ball_1)\cap f(L)\subset f(\Ball 2),
\]
for every $\rho\geq 0$, $f$ linear function and $L\in\L^n_i$, $1\leq i\leq n$.
From this it immediately follows the lemma.
\end{proof}

\begin{theorem}\label{th:Minkowskispace}
Let $K\in\K^n$ in a Minkowski space $(\R^n,||\,\cdot\,||)$ of unit ball $\Ball$. Then
\[
\frac{\cir_{n-i+1}(K,\Ball)}{\inr_i(K,\Ball)}\leq n\sqrt{n},\quad 1\leq i\leq n.
\]
If $\Ball=\Ball_2$ or $K\in\K^n_0$, the bound becomes $n$.
Moreover, if both occur, the bound further reduces to $\sqrt{n}$.
\end{theorem}

\begin{proof}
After suitable translations of $K$ and $\Ball$, let $\E_K$ and $\E_\Ball$
be the ellipsoids of John of $K$ and $\Ball$, respectively. We therefore have that
$\E_K\subset K\subset \rho_K\E_K$ and $\E_\Ball\subset \Ball\subset \rho_\Ball\E_\Ball$,
where $\rho_K$ is either $n$, or $\sqrt{n}$ if $K\in\K^n_0$, whereas
$\rho_\Ball$ is either $\sqrt{n}$, or $1$ if $\Ball=\Ball_2$. Then
\[
\frac{\cir_{n-i+1}(K,\Ball)}{\inr_i(K,\Ball)}\leq
\rho_K\rho_\Ball\frac{\cir_{n-i+1}(\E_K,\E_\Ball)}{\inr_i(\E_K,\E_\Ball)}.
\]

Let $f$ be a linear application s.t.~$f(\E_\Ball)=\Ball_2$. Lemma \ref{l:affine} implies that
\[
\frac{\cir_{n-i+1}(K,\Ball)}{\inr_i(K,\Ball)}\leq
\rho_K\rho_\Ball\frac{\cir_{n-i+1}(f(\E_K),\Ball_2)}{\inr_i(f(\E_K),\Ball_2)},
\]
and finally, since $f(\E_K)$ is an ellipsoid, then
$\cir_{n-i+1}(f(\E_K),\Ball_2)=\inr_i(f(\E_K),\Ball_2)$
from which we conclude the result.
\end{proof}

Theorem \ref{th:Minkowskispace} raises the question whether the
estimates are tight or not, and how far they are from being best possible.

\emph{Acknowledgement}. I would like to thank Ren\'e Brandenberg and
Mar\'ia A.~Hern\'andez Cifre for fruitful discussions
and comments on the subject, as well as many useful advices
and proofreadings during the writing of this paper.

I would also like to thank the anonymous referee for his useful comments and suggestions,
which improved the paper.

\end{document}